\documentclass[11pt]{amsart}
\usepackage{graphicx,pict2e}
\usepackage{amssymb,amsmath}
\usepackage{amsthm,mathabx}
\usepackage{qsymbols}
\usepackage{latexsym,enumitem,diagbox}

\usepackage{dsfont}

\newtheorem{thm}{Theorem}
\newtheorem{lem}[thm]{Lemma}
\newtheorem{prop}[thm]{Proposition}
\newtheorem{cor}[thm]{Corollary}

\newtheorem{rem}[thm]{Remark}

\newcommand{\Z}{\mathbb{Z}}
\newcommand{\R}{\mathbb{R}}
\newcommand{\1}{\mathds{1}}
\renewcommand{\P}{\mathbb{P}}

\newcommand{\ep}{\varepsilon}

\newcommand{\Var}{\operatorname{Var}}

\begin{document}
\title{Transport Proofs of some discrete variants of the Pr\'ekopa-Leindler inequality}

\author[N. Gozlan, C. Roberto, P.-M. Samson, P. Tetali]{Nathael Gozlan, Cyril Roberto, Paul-Marie Samson, Prasad Tetali}

\date{\today}

\thanks{This research is partly funded by the B\'ezout Labex, funded by ANR, reference ANR-10-LABX-58 and the  Labex MME-DII funded by ANR, reference ANR-11-LBX-0023-01. Research of P.T. is supported in part by the NSF grant DMS-1811935}

\address{N. Gozlan : Universit\'e Paris Descartes, MAP5, UMR 8145, 45 rue des Saints Pères, 75270 Paris Cedex 06}
\address{P.-M. Samson : Universit\'e Paris-Est, Laboratoire d'Analyse et de Math\'ematiques Appliqu\'ees (UMR 8050), UPEM, UPEC, CNRS, F-77454, Marne-la-Vall\'ee, France}

\address{C. Roberto : Universit\'e Paris Nanterre - Modal'X, 200 avenue de la R\'epublique 92000 Nanterre, France}
\address{P. Tetali : School of Mathematics \& School of Computer Science, Georgia Institute of Technology,
Atlanta, GA 30332}
\email{natael.gozlan@parisdescartes.fr, croberto@math.cnrs.fr, paul-marie.samson@u-pem.fr,\linebreak tetali@math.gatech.edu}


\begin{abstract}
We give a transport proof of a discrete version of the displacement convexity of  entropy on integers ($\Z$), and get, as a consequence, two discrete forms of the Pr\'ekopa-Leindler Inequality : the Four Functions Theorem of Ahlswede and Daykin on the discrete hypercube \cite{AD} and a recent result on $\Z$ due to Klartag and Lehec  \cite{KL}.
\end{abstract}

\maketitle
\maketitle

\section*{Introduction}
The aim of the paper is to develop a transport approach to some discrete versions of the Pr\'ekopa-Leindler Inequality \cite{Pre1,Pre2,Lei}, namely the Four Functions Theorem due to  Ahlswede and Daykin \cite{AD} and a recent result of Klartag and Lehec \cite{KL} on $\Z$. Both inequalities will be a consequence of the stronger displacement convexity of entropy on the set of integers. Before presenting these discrete functional inequalities, let us recall the original continuous statement inspiring them. 

The classical Pr\'ekopa-Leindler Inequality is the following.
\begin{thm}[Pr\'ekopa-Leindler]\label{thm:PL}
Suppose that  $f,g,h : \R^n \to \R^+$ are measurable functions such that, for some $t \in (0,1)$, 
\begin{equation}\label{eq:PLhyp}
f(x)^{1-t}g(y)^t \leq h((1-t)x+ty),\qquad \forall x,y \in \R^n.
\end{equation}
Then
\[
\left(\int_{\R^n} f(x)\,dx\right)^{1-t} \left(\int_{\R^n} g(y)\,dy\right)^{t} \leq \int_{\R^n} h(z)\,dz .
\]
\end{thm}
The Pr\'ekopa-Leindler Inequality is a functional version of the celebrated Brunn-Minkowski Inequality stating that for all Borel sets $A,B \subset \R^n$ and all $t \in (0,1)$ it holds
\[
\mathrm{Vol}((1-t)A+tB) \geq \mathrm{Vol}(A)^{1-t} \mathrm{Vol}(B)^t,
\]
where $\mathrm{Vol}(\,\cdot\,)$ denotes the Lebesgue measure on $\R^n.$
It is more generally intimately related to the study of log-concave measures which is of considerable importance in convex geometry, probability theory and statistics. In particular, many geometric and functional inequalities for uniformly log-concave probability  measures can be derived from Theorem 1 (see in particular the paper \cite{BL} by Bobkov and Ledoux). We refer to \cite{Gar} for a thorough presentation of the subject as well as for historical comments on Theorem \ref{thm:PL}. 

The question of extending the Pr\'ekopa-Leindler inequality outside the flat space framework has been tackled by many authors in recent years and turned out to be extremely fruitful in Geometry, Analysis and Probability. A first step has been accomplished by Cordero-Erausquin, McCann and Schmuckenschl\"ager in \cite{CEMS1,CEMS2}, who obtained extensions of the Pr\'ekopa-Leindler inequality on Riemannian manifolds with a lower bounded Ricci curvature.  Their extension is closely related to displacement convexity properties of entropic functionals, first introduced by McCann in \cite{McC} in the flat space framework, and then extended to Riemannian manifolds by Otto and Villani \cite{OV} and von Renesse and Sturm \cite{vRS}. This displacement convexity formulation is actually equivalent to lower bounds on the Ricci curvature and led to the Lott-Sturm-Villani \cite{LV, Stu1,Stu2} definition of metric measure spaces with lower bounded Ricci curvature which makes sense even in a non-smooth framework. 

In a similar vein, it would also be satisfactory to extend the Pr\'ekopa-Leindler inequality to discrete frameworks such as graphs (which are not covered by the Lott-Sturm-Villani theory). Several general definitions of discrete spaces with lower bounded curvature were recently proposed, in particular by Bonciocat and Sturm \cite{BS09}, Ollivier \cite{Oll}, Ollivier and Villani \cite{OV12}, Erbar and Maas \cite{EM}, Hillion \cite{Hil} or the authors \cite{GRST}. While these different definitions are all efficient at the level of functional inequalities and are satisfied by a large collection of classical graphs, none of them really succeeds in leading to a satisfactory Pr\'ekopa-Leindler or Brunn-Minkowski inequality on those spaces. 

However, for at least two specific discrete spaces, convincing Pr\'ekopa-Leindler type inequalities already exist. 

The first one, is the celebrated Four Functions Theorem on the discrete hypercube $\{0,1\}^n$ by Ahlswede and Daykin \cite{AD}. To recall its statement, we will need the following notation. The discrete hypercube will be denoted by $\Omega_n := \{0,1\}^n$ and for all $x=(x_1,\dots,x_n), \ y=(y_1,\dots,y_n) \in \Omega_n$, one defines 
\[x \wedge y :=(\min(x_1,y_1),\dots,\min(x_n,y_n)) \quad \mbox{ and } \quad x \vee y :=(\max(x_1,y_1),\dots,\max(x_n,y_n))\,.\] 
\begin{thm}[Ahlswede-Daykin]\label{thm:4FT}
Suppose that $f, g, h, k \colon \Omega_n \to \R^+$ are such that
\[
f(x)g(y) \leq h(x  \wedge y)k(x \vee y), \qquad \forall x,y \in \Omega_n\,,
\]
then
\[
\sum_{x \in \Omega_n}  f(x)  \sum_{x\in \Omega_n} g(x)  \leq \sum_{x \in \Omega_n}  h(x) \sum_{x \in \Omega_n} k(x).
\]
\end{thm}
Note that this result mimics the statement of Theorem \ref{thm:PL} for $t=1/2$ on $\Omega_n$. Theorem \ref{thm:4FT} has important implications in terms of correlation inequalities, as it gives back in particular the classical FKG inequality which has a lot of applications in percolation and statistical mechanics \cite{FKG}. 

The second discrete form of the Pr\'ekopa-Leindler Inequality we will consider is a recent one due to Klartag and Lehec \cite{KL}, and holds on the space $\Z$ of integers. Denote by $\lceil \cdot \rceil$ and $\lfloor \cdot \rfloor$ the ceiling and floor functions respectively.
\begin{thm}\label{thm:KL}
Suppose that $f,g,h,k : \Z\to \R^+$ are such that
\begin{equation}\label{eq:hyp}
f(x)g(y) \leq h\left(\left\lfloor \frac{x+y}{2}\right \rfloor\right)k\left(\left\lceil \frac{x+y}{2} \right\rceil\right),\qquad \forall x,y \in \Z.
\end{equation}
Then 
\[
\left(\sum_{x\in \Z} f(x)\right)\left(\sum_{y\in \Z} g(y)\right)\leq \left(\sum_{x\in \Z} h(x)\right)\left(\sum_{y\in \Z} k(y)\right).
\]
\end{thm}
As we will see in Section \ref{Sec2}, Theorem \ref{thm:KL} implies Theorem \ref{thm:4FT} for $n=1$ (which then gives the full conclusion by induction, see the proof of Theorem \ref{thm:4FT} in Section \ref{Sec1}). Moreover Theorem~\ref{thm:KL} implies back Theorem \ref{thm:PL} for $t=1/2$ (and thus for all other values of $t$). The proof given by Klartag and Lehec in \cite{KL} relies on rather sophisticated tools of stochastic analysis on the Poisson space and in particular on a stochastic representation formula for the relative entropy functional with respect to the Poisson distribution on the (non-negative) integers.

As already stated above, the main objective of the present paper is to recover Theorems~\ref{thm:4FT} and \ref{thm:KL} by means of optimal transport tools.  In the continuous setting, optimal transport is indeed a very efficient way to establish functional inequalities (see \cite{Vil1,Vil2} and the references therein) and it is a challenging question to see how these powerful techniques can be adapted to the discrete world. To make this introduction more self-contained and to illustrate the difficulties in dealing with discrete structures, let us briefly recall a classical transport proof of Theorem \ref{thm:PL} in dimension $1$.
\proof[Proof of Theorem~\ref{thm:PL} for $d=1$] 
Without loss of generality, one can assume that $\int_{\R} f(x)\,dx = \int_{\R} g(y)\,dy=1$, with $f$ and $g$ two positive and continuous functions. Defining $\mu(dx) = f(x)\,dx$, $\nu(dy)= g(y)\,dy$, a natural transport map between the probability measures $\mu$ and $\nu$ is given by $T(x) = F_{\nu}^{-1}\circ F_\mu(x)$, where $F_\mu(x) = \int_{-\infty}^x f(u)\,du$, $x\in \R$, and $F_{\nu}(y) = \int_{-\infty}^y g(v)\,dv$, $y\in \R$, are the cumulative distribution functions of $\mu$ and $\nu$. The change of variable formula immediately gives the following relation between $f$ and $g$:
\begin{equation}\label{eq:Monge-Ampere}
f(x)  = g(T(x))T'(x),\qquad \forall x \in \R.
\end{equation}
Plugging $y = T(x)$ into \eqref{eq:PLhyp} one gets by change of variables ($z=(1-t)x+tT(x)$, note that $T$ is increasing by construction)
\begin{align*}
\int_{\R} h(z)\,dz & = \int_{\R} h((1-t)x + t T(x)) [(1-t) + t T'(x)]\,dx\\
& \geq  \int_{\R} f(x)^{1-t} g(T(x))^t T'(x)^t\,dx\\
& =  \int_{\R} f(x)^{1-t} f(x)^t\,dx = 1,
\end{align*}
where the inequality comes from  \eqref{eq:PLhyp} and the arithmetic-geometric inequality $(1-t)a+tb \geq a^{1-t} b^t$, $a,b\geq0$, $t \in [0,1]$ (appplied to $a=1$ and $b=T'(x)$), while the last equality comes from \eqref{eq:Monge-Ampere}.
\endproof
The proof for $n \geq 2$ is done by induction (see \textit{e.g} the proof of \cite[Theorem 2.13]{Led}). It is also possible to prove this result directly in dimension $n$, by using the Brenier or the Knothe transport maps and the Monge-Amp\`ere equation. See \cite[Chapter 6]{Vil1} for details. Note that the use of coupling arguments for establishing Brunn-Minkowski type inequalities goes back at least to Knothe \cite{Kno}.

Analyzing the proof above immediately reveals two obvious obstacles that prevent to export it easily to the discrete setting:
\begin{enumerate}
\item Transport \emph{maps} between probability measures $\mu$ and $\nu$ usually do not exist when the space is discrete and one often needs to cut the mass of atoms of the source measure $\mu$ to reconstruct the target measure $\nu$;
\item Even if there is a transport map $T$ sending $\mu$ on $\nu$, there is no Jacobian equation such as \eqref{eq:Monge-Ampere}.
\end{enumerate}

In the case of Theorem \ref{thm:4FT} and \ref{thm:KL}, it turns out that these difficulties can be circumvented. 
It would be useless at this point to state general rules, however it seems at least that in both situations choosing $t=1/2$ helps a lot by introducing symmetry and compensations to overcome the lack of Jacobian equation. 

In fact, we will go beyond Theorem \ref{thm:4FT} and \ref{thm:KL} by proving, by transport arguments, a stronger statement: namely an entropic version of the Pr\'ekopa-Leindler Inequality (that we may also call displacement convexity of entropy), see Theorem \ref{thm:KLentropic} for a precise statement. In that sense, since such an entropic statement implies the Klartag-Lehec version of the Pr\'ekopa-Leindler Inequality on $\mathbb{Z}$, which in turn, at the price of an obvious induction step, implies the Four Functions theorem, all results appear to be the consequence of one single (transport) proof. Moreover, our displacement convexity result on the  integers, as the mesh size goes to 0, converges to the classical displacement convexity  of entropy on the line (for $t=1/2$), obtained by McCann in \cite{McC} which shows the compatibility of our results to the well-known equivalent statement in the continuous.

\bigskip

The paper is organized as follows.

In Section 1, we give a simple proof of Theorem \ref{thm:4FT}, based on the construction of an explicit coupling in dimension $n=1$ and on the dual formulation of the relative entropy functional. As already explained, Theorem \ref{thm:4FT} can also be seen as a consequence of Theorem  \ref{thm:KL}. However, the proof is very simple and it seemed to us that it nicely illustrates the power of the transport techniques in discrete and therefore it is worth a separate presentation. Then we show how to recover a significant part of the classical Pr\'ekopa-Leindler inequality from Theorem \ref{thm:4FT}, passing from the discrete to the continuous by means of the Central Limit Theorem. 

In Section 2, we prove a stronger entropic version of Theorem \ref{thm:KL}, namely Theorem \ref{thm:KLentropic}, based on the one-dimensional monotone rearrangement coupling. We also show how to fully recover the Pr\'ekopa-Leindler inequality starting from Theorem \ref{thm:KL}, again passing from discrete to continuous, but here using instead that the mesh size of the grid shrinks to $0$.  

Finally, Section 3 is devoted to curved versions of Theorem \ref{thm:KL} applying to probability measures with a log-concave probability mass function.

\section{The Four Functions theorem}\label{Sec1}
\subsection{A transport proof of the Four Functions Theorem}
In the following, we prove the Four Functions Theorem using transport ingredients and a duality formula. 

We will use the following notations. The set of all probability measures on $\Omega_n = \{0,1\}^n$ will be denoted by $\mathcal{P}(\Omega_n)$ and the set of functions on $\Omega_n$ by $\mathcal{F}(\Omega_n)$. For all $a \in \Omega_1$ and $h \in \mathcal{F}(\Omega_n)$, the function $h^a: \Omega_{n-1} \to \R$ is defined by
\[
h^a(x)= h(x,a),\qquad \forall x \in \Omega_{n-1}.
\]
For convenience, we restate the Ahlswede-Daykin Theorem with an additive hypothesis (which corresponds to Theorem \ref{thm:4FT} with $f=e^{h_1}$, $g=e^{h_2}$, $h=e^{h_3}$ and $k=e^{h_4}$).
\begin{thm} \label{thm:4FTgen}
Let $n \geq 1$. Suppose that $h_1$, $h_2$, $h_3$, $h_4 \colon \Omega_n \to \R$ are such that
\[
h_1(x) + h_2(y) \leq h_3(x  \wedge y) + h_4(x \vee y), \qquad \forall x,y \in \Omega_n.
\]
Then
\[
\sum_{x \in \Omega_n} e^{h_1(x)} \sum_{x \in \Omega_n} e^{h_2(x)} \leq \sum_{x \in \Omega_n} e^{h_3(x)}\sum_{x \in \Omega_n} e^{h_3(x)} .
\]
\end{thm}

Recall the following duality formula involving the relative entropy functional. Let $m_n$ be the uniform measure on $\Omega_n$ and define for all probability measures $\nu$ on $\Omega_n$ 
\[
H(\nu | m_n) = \int \log \left(\frac{d\nu}{dm_n}\right)\,d\nu .
\]
Then, for any function $f : \Omega_n \to \R$, it holds
\begin{equation} \label{eq:dual}
\log \int e^{f}\,dm_n = \sup_{\nu \in \mathcal{P}(\Omega_n)} \left\{\int f\,d\nu - H(\nu|m_n)\right\} .
\end{equation}

In the proof of Theorem \ref{thm:4FTgen} we will also use the following coupling lemma whose proof is  elementary. We recall that if $\nu_1,\nu_2$ are two probability measures on a measurable space $(E,\mathcal{A})$, a coupling of $\nu_1$ and $\nu_2$ (in that order) is a probability measure $\pi$ on the product space $E\times E$ having $\nu_1$ as first marginal and $\nu_2$ as second marginal, that is to say such that 
\[
\pi(A\times E) = \nu_1(A)\qquad \text{and}\qquad \pi(E\times B) = \nu_2(B)
\]
for all $A,B \in \mathcal{A}.$ Recall also that that if $\mu$ is a probability measure on $(E,\mathcal{A}) $ and $S : E \to F$ a measurable map taking values in another measurable space $(F,\mathcal{B})$, then the image of $\mu$ under the map $S$ (or push forward of $\mu$ under the map $S$) is the probability measure denoted by $S_\#\mu$ defined as $S_\#\mu (B) = \mu(S^{-1}(B))$, $B \in \mathcal{B}$.

\begin{lem} \label{transport}
Let $\nu_1 , \nu_2 \in \mathcal{P}(\Omega_1)$ and set $S \colon \Omega_1^2 \ni (x,y) \mapsto (x \wedge y, x \vee y)$.
\begin{itemize}
\item[$(i)$]  if $\nu_2(0) \leq \nu_1(0)$ then there exists a (unique) coupling $\pi$  of $\nu_1$ and  $\nu_2$ such that
$\widetilde \pi : = S \sharp \pi$ is also a coupling of $\nu_1$ and $\nu_2$.
Moreover in this case $\pi = \widetilde\pi$ and $\pi(0,0)=\nu_2(0)$, $\pi(1,0)=0$, $\pi(0,1)=\nu_1(0)-\nu_2(0)$ and
$\pi(1,1)=\nu_1(1)$.
\item[$(ii)$]
if $\nu_2(0) \geq \nu_1(0)$ then there exists a (unique) coupling $\pi$ of $\nu_1$ and $\nu_2$ such that $\widetilde \pi = S \sharp \pi$ is a coupling  of $\nu_2, \nu_1$. Moreover $\pi(0,0)=\widetilde \pi(0,0)=\nu_1(0)$, $\pi(1,1)=\widetilde \pi(1,1)=\nu_2(1)$,
$\pi(0,1)=\widetilde \pi(1,0) = 0$ and $\pi(1,0)=\widetilde \pi(0,1) = \nu_2(0) - \nu_1(0)$. 
\end{itemize}
\end{lem}
\begin{rem}
The coupling $\pi$ in $(i)$ (resp. $(ii)$) is nothing but the non-decreasing (non-increasing) rearrangement coupling.

The above lemma is very much one-dimensional. In fact, it is easy to construct examples of measures $\nu_1 , \nu_2 \in \mathcal{P}(\Omega_n)$, for $n \geq 2$, such that there does not exist any coupling $\pi$  of $\nu_1$ and  $\nu_2$ such that
$\widetilde \pi : = S \sharp \pi$ (with $S$ that acts coordinate by coordinate) is a coupling of $\nu_1$ and $\nu_2$ or a coupling of $\nu_2$ and $\nu_1$.
\end{rem}

\begin{proof}
We will first prove Item $(i)$.
In the following diagram we represent the couplings $\pi$ on the left, and $\widetilde \pi$ on the right, with their marginals.
$$
\begin{tabular}{|c||c|c||c|}
\hline
\backslashbox{$x$}{$y$} 
& 0 & 1 & \\
\hline
 $0$ & $\pi(0,0)$ & $\pi(0,1)$ & $\nu_1(0)$ \\
\hline
 $1$ & $\pi(1,0)$ & $\pi(1,1)$ & $\nu_1(1)$ \\
 \hline
  & $\nu_2(0)$ & $\nu_2(1)$ & \\
  \hline
\end{tabular}
\quad \stackrel{S}{\longrightarrow}
\quad 
\begin{tabular}{|c||c|c||c|}
\hline
\backslashbox{$x\wedge y$}{$x \vee y$} 
& 0 & 1 & \\
\hline
 $0$ & $\pi(0,0)$ & $\pi(0,1) + \pi(1,0)$ & $\nu_1(0)$ \\
\hline
 $1$ & $0$ & $\pi(1,1)$ & $\nu_1(1)$ \\
 \hline
  & $\nu_2(0)$ & $\nu_2(1)$ & \\
  \hline
\end{tabular}
$$
Once one observes that necessarily $\widetilde \pi(1,0)=0$ (since there do not exist $x,y \in \Omega_1$
with $x \wedge y =1$ and $x \vee 1=0$), and $\widetilde \pi(0,0)= \pi(0,0)$ and 
$\widetilde \pi(1,1)= \pi(1,1)$, then all the values of $\widetilde \pi(i,j)$ and $\pi(i,j)$ can be deduced from the marginals (details are left to the reader).
A similar reasoning leads to the conclusion of Item $(ii)$. The uniqueness part is obvious from the construction.
\end{proof}

\begin{proof}[Proof of Theorem \ref{thm:4FTgen}]
The proof goes by induction on $n\ge 1$. We will prove the base case towards the end of the proof. Assume first  that the result holds on $\Omega_{n-1}$. Then choose
four functions $h_1, h_2, h_3, h_4 \colon \{0,1\}^n \to \mathbb{R}$ satisfying
\begin{equation} \label{condition}
h_1(x) + h_2(y)  \leq h_3(x  \wedge y) + h_4(x \vee y), \qquad \forall x,y \in \Omega_n .
\end{equation}
Fix $a, b \in \{0,1\}$ ; applying Condition \eqref{condition} to $x=(x_1',\dots,x_{n-1}',a)$ and $y=(y_1',\dots,y_{n-1}',b)$ we get that
\[
h_1^a(x') + h_2^b(y') \leq h_3^{a \wedge b}(x'  \wedge y') + h_4^{a \vee b}(x' \vee y'), \qquad \forall x',y' \in \Omega_{n-1} 
\]
which is precisely the condition of the theorem in dimension $n-1$ for the four functions 
$h_1^{a}, h_2^{b}, h_3^{a \wedge b}$ and $h_4^{a \vee b}$. Applying the induction hypothesis we conclude that
\[
\log \left( \sum_{x \in \Omega_{n-1}}e^{h_1^a(x)} \right) + \log \left(  \sum_{x \in \Omega_{n-1}}e^{h_2^b(x)} \right)
\leq 
\log \left(  \sum_{x \in \Omega_{n-1}}e^{h_3^{a \wedge b}(x)} \right) + \log \left( \sum_{x \in \Omega_{n-1}}e^{h_4^{a \vee b}(x)} \right) .
\]
The latter holds for all $a, b \in \Omega_1$. Hence, if we set $H_i(a) := \log \left( \sum_{x \in \Omega_{n-1}}e^{h_i^a(x)}\right)$,
for $i \in \{1,2,3,4\}$, we have
\[
H_1(a) + H_2(b) \leq H_3(a \wedge b) +  H_4(a \vee b) \qquad \forall a, b \in \Omega_1 .
\]
Now applying the result on $\Omega_1$, we conclude that
\[
\log \left( \sum_{x \in \Omega_1} e^{H_1(x)} \right) + \log \left( \sum_{x \in \Omega_1} e^{H_2(x)} \right)  
\leq 
\log \left( \sum_{x \in \Omega_1} e^{H_3(x)} \right)  + \log \left( \sum_{x \in \Omega_1} e^{H_4(x)} \right) \,.
\]
This leads to the desired conclusion since, by construction, for all $i\in \{1,2,3,4\}$ it holds $\log \left( \sum_{x \in \Omega_1} e^{H_i(x)} \right)
= \log \left( \sum_{x \in \Omega_n} e^{h_i(x)} \right) $.

Hence, in order to conclude the proof we need to prove the theorem on $\Omega_1$. 
To that purpose, fix four functions $h_1, h_2, h_3, h_4 \colon \Omega_1 \to \mathbb{R}$ satisfying Condition \eqref{condition} (with $n=1$) and let $\nu_1,\nu_2 \in \mathcal{P}(\Omega_1)$. Let us show that
\begin{equation}\label{eq:majoration}
\left(\int h_1\, d\nu_1-H(\nu_1|m_1)\right)  +  \left(\int h_2\, d\nu_2 -H(\nu_2|m_1)\right)  
\leq 
\log \left( \sum_{x \in \Omega_1} e^{h_3(x)} \right) + \log \left( \sum_{x \in \Omega_1} e^{h_4(x)} \right).
\end{equation}

First assume that $\nu_1(0) \leq \nu_2(0)$. Thanks to Item $(i)$ of Lemma \ref{transport} above, there exists a coupling $\pi$ of $\nu_1$ and $\nu_2$ such that the coupling $\widetilde\pi$ defined as the push forward of $\pi$ under the map $S:\Omega_1^2 \ni (x,y) \mapsto (x \wedge y, x \vee y)$ is still a coupling of $\nu_1$ and $\nu_2$. 
It follows from the very definition of the coupling, from Condition \eqref{condition}, and by definition of the push-forward, that
\begin{align}
\int h_1\, d\nu_1  +  \int h_2\, d\nu_2 
& =
\int_{\Omega_1^2} [h_1(x) + h_2(y)] \,d\pi(x,y)   
 \leq
\int_{\Omega_1^2} [h_3(x\wedge y) + h_4(x \vee y)] \,d\pi(x,y)    \label{eq:coupling} \\
& = 
\int_{\Omega_1^2} h_3(x) + h_4(y) \,d\widetilde\pi(x,y)    
 =
\int h_3 \,d\nu_1  +\int h_4 \,d\nu_2  . \nonumber
\end{align}
Therefore, by \eqref{eq:dual},
\begin{align*}
& \left(\int h_1\, d\nu_1-H(\nu_1|m_1)\right)  +  \left(\int h_2\, d\nu_2 -H(\nu_2|m_1)\right) 
\leq 
\left(\int h_3 \,d\nu_1 - H(\nu_1|m_1) \right) \\
& \phantom{AAAAAAAAAAAA}  + \left( \int h_4\, d\nu_2 - H(\nu_2|m_1) \right)
\leq 
\log \left( \sum_{x \in \Omega_1} e^{h_3(x)} \right) + \log \left( \sum_{x \in \Omega_1} e^{h_4(x)} \right),
\end{align*}
which proves \eqref{eq:majoration} in this case.

Now, if $\nu_1(0) > \nu_2(0)$, then according to Item $(ii)$ of Lemma \ref{transport}, there exists a coupling $\pi$ of $\nu_1$ and $\nu_2$  such that the probability $\widetilde{\pi} = S_\#\pi$ is now a coupling of $\nu_2$ and $\nu_1$ (in that order). Therefore, reasoning exactly as in \eqref{eq:coupling}, one gets $\int h_1\, d\nu_1  +  \int h_2\, d\nu_2 \leq \int h_3 \,d\nu_2  +\int h_4 \,d\nu_1$, from which one concludes that \eqref{eq:majoration} holds also in this case.

Finally, taking the supremum over $\nu_1$ and $\nu_2$ in \eqref{eq:majoration} gives , thanks to \eqref{eq:dual},
\[
\log \left( \sum_{x \in \Omega_1} e^{h_1(x)} \right) + \log \left( \sum_{x \in \Omega_1} e^{h_2(x)} \right)  
\leq 
\log \left( \sum_{x \in \Omega_1} e^{h_3(x)} \right)  + \log \left( \sum_{x \in \Omega_1} e^{h_4(x)} \right) \,.
\]
and completes the proof of Theorem \ref{thm:4FTgen}.
\end{proof}

A careful reading of the proof of Theorem \ref{thm:4FTgen} actually leads to a slightly more general result that we now describe. Consider a functional $\Phi$ on $\mathcal{F}(\Omega_1)$ and assume that it can be written as follows
\begin{equation} \label{dual}
\Phi(h) = \sup_{\nu \in \mathcal{P}(\Omega_1)} \left\{ \int h\,d\nu - \Psi(\nu)\right\},\qquad h \in \mathcal{F}(\Omega_1),
\end{equation}
where $\Psi : \mathcal{P}(\Omega_1) \to \R \cup \{\infty\}$ is a given function. Then, we define by induction a sequence of functions $\Phi^n$ on $\mathcal{F}(\Omega_n)$ as follows: $\Phi^1 = \Phi$ and for all $n\geq 2$, 
\[
\Phi^{n} (h) = \Phi (a \mapsto \Phi^{n-1} (h^a)),\qquad h \in \mathcal{F}(\Omega_n),
\]
where we recall that for all $a\in \Omega_1$ and $h \in \mathcal{F}(\Omega_n)$, the function $h^a: \Omega_{n-1} \to \R$ is defined by
$h^a(x)= h(x,a)$, $x \in \Omega_{n-1}$.

Following the exact same proof of Theorem \ref{thm:4FTgen} (details of which are left to the reader), we can conclude that, if $h_1$, $h_2$, $h_3$, $h_4 \colon \Omega_n \to \R$ are such that
\[
h_1(x) + h_2(y) \leq h_3(x  \wedge y) + h_4(x \vee y), \qquad \forall x,y \in \Omega_n,
\]
then
\[
\Phi^n(h_1) + \Phi^n(h_2) \leq \Phi^n(h_3) + \Phi^n(h_4) .
\]

This is a generalization of Theorem \ref{thm:4FTgen} since the relative entropy $\Psi(\nu)=H(\nu|m_n)$ leads to $\Phi(h)=\log (\int_{\Omega_1} e^h dm_1)$ by \eqref{eq:dual}, and therefore, by a straightforward induction, to $\Phi^n(h)=\log (\int_{\Omega_n}e^h dm_n)$.  However, we could not find any other explicit example of functional $\Phi$ and $\Phi^n$ of real interest.
One of the reasons can be found in Hardy, Littlewood and Polya \cite[Chapter 3]{HLP}. Indeed, studying the generalized mean $F^{-1}(\int F(h)dm_1)$, these authors prove that, under some mild assumptions, it must be that $F(x)=\kappa e^{cx}$ for some constants $\kappa,c$, leading back to the previous example.

Another natural example may be given by $\Psi(\nu)=+\infty$ for all $\nu$ expect one measure, say $m_1$, for which $\Psi(m_1)=0$. Then, $\Phi(h)=\int hdm_1$ and therefore $\Phi^n(h)=\int h dm_1^{\otimes n}$, where $m_1^{\otimes n}$ is the $n$-fold product of $m_1$, \textit{i.e.}\ $m_1^{\otimes n}=m_n$. In that case,  the conclusion above is nontrivial though being a consequence of the classical conclusion of the four functions theorem (by considering $\varepsilon h_i$ in the limit $\varepsilon \to 0$).

A further generalization may be as follows.
Let $U \colon [0,\infty) \to \mathbb{R}$ denote a semi-continuous, strictly convex function satisfying $\lim_{x \to \infty} U(x)/x = \infty$ and  $U(1) \geq 0$. Then, given $\mu,\nu \in \mathcal{P}(\Omega_n)$,  we set
$U_\mu(\nu) = \int U(f)d\mu$,
if $\nu$ is absolutely continuous with respect to $\mu$ with density $f$, and $U_\mu(\nu) = +\infty$ otherwise. With such a definition, the special choice $U(x)=x\log x$ amounts to $U_\mu(\nu)=H(\nu|\mu)$. Furthermore, since $U(1)\geq0$, by Jensen's inequality $U_\mu(\nu) \geq 0$ for all $\nu \in \mathcal{P}(\Omega_n)$. Also, for any $f \colon \{0,1\}^n \to \mathbb{R}$ and $\mu \in \mathcal{P}(\Omega_n)$, set
$\Lambda_\mu(f):=\sup_{\nu \in \mathcal{P}(\Omega_n)} \left( \int_{\Omega_n} f d\nu - U_\mu(\nu) \right)$
which generalizes \eqref{eq:dual}. For such $U$'s, as proved in \cite[Proposition 2.9]{GRS}, it holds
$$
\Lambda_\mu(f) = \inf_{t \in \mathbb{R}} \left\{ \int [U^*(f+t)-t ]d\mu \right\}
$$
and
$$
U_\mu(\nu)= \sup_{f} \left\{ \int fd\nu - \Lambda_\mu(f) \right\} = \sup_{f} \left\{ \int fd\nu - \int U^*(f)d\mu \right\}
$$
with $U^*(y):=\sup_{x >0} \{xy-U(x)\}$, $y \in \mathbb{R}$. For instance, the choice $U(x)=x^2/2$, $x \geq 0$ leads to $\Lambda_{m_1}(f)=\Var_{m_1}(f)+\int f dm_1 - \frac{1}{2}$ if $f(0)-f(1) \in [-2,2]$ and $\Lambda_{m_1}(f)= \max(f(0),f(1))-1$ otherwise. 
At the price of multiplying $h_i$ by a constant, we can assume that $\max h - \inf \leq 2$ so that $\Phi(h)= \Var_{m_1}(h)+\int h dm_1 - \frac{1}{2}$ is explicit so that one can, at least theoretically, express $\Phi^n$ in this case.

\subsection{From the Four Function Theorem to the Pr\'ekopa-Leindler Inequality}
Using the Four Functions Theorem, we shall prove the following weak version of the Pr\'ekopa-Leindler Inequality. We state and prove the result in dimension one, for simplicity, but it holds in any dimension with no extra complication besides presentation.

\begin{prop} \label{PL}
Let $f, g, h \colon \mathbb{R} \to \mathbb{R}$ be three continuous functions satisfying 
\[
\frac{1}{2}f(x) + \frac{1}{2}g(y) \leq h\left(\frac{x+y}{2}\right) \qquad \forall x,y \in \mathbb{R} .  
\] 
Assume furthermore that $h$ is convex and bounded from below. Then, it holds
\[
\left(\int_{\mathbb{R}} e^{f(x)} \,dx \right)^{1/2} \left(\int_{\mathbb{R}} e^{g(y)}\,dy \right)^{1/2} \leq \int_{\mathbb{R}} e^{h(z)} \,dz.
\]
\end{prop}
It should be noticed that equality cases are known in the Pr\'ekopa-Leindler inequality \cite{Dub} and correspond to choosing precisely  $h$ convex, and $f$ and $g$ proper translation and dilation of $h$. Of course, the extra assumptions of continuity of $f,g$ and lower boundedness of $h$ could be removed via standard approximation arguments, but we refrain from further discussion, since it does not seem possible to remove the convexity assumption on $h$ and to recover the full conclusion of Theorem \ref{thm:PL}.

\begin{proof}
Let $f, g, h \colon \mathbb{R} \to \mathbb{R}$ be continuous functions satisfying 
\[
\frac{1}{2}f(x) + \frac{1}{2}g(y) \leq h\Big(\frac{x+y}{2}\Big) \qquad \forall x,y \in \mathbb{R}\,,  
\]
with $h$ convex and bounded from below. First let us assume that $f$ and $g$ are bounded from above. For any $n$, define the following three functions on $\Omega_n$: for $x=(x_1,\dots,x_n) \in \Omega_n$,
set
\[
F_n(x):= f \left( \frac{\sum_{i=1}^n x_i - \frac{n}{2}}{\sqrt{n}/2} \right),
\quad
G_n(x):= g \left( \frac{\sum_{i=1}^n x_i - \frac{n}{2}}{\sqrt{n}/2} \right)
\;\; \mbox{and} \;\;
H_n(x):= h \left( \frac{\sum_{i=1}^n x_i - \frac{n}{2}}{\sqrt{n}/2} \right) .
\]
Then we observe that, for any $x,y  \in \Omega_n$, coordinate-wise
\[
x+y = x \wedge y + x \vee y.
\] 
Hence, the condition satisfied by $f, g$ and $h$ transfers to $F_n, G_n$ and $H_n$ as follows: for all $x, y \in \Omega_n$,
\[
F_n(x) + G_n(y) \leq 2H_n\left( \frac{x \wedge y + x \vee y}{2} \right) \leq H_n(x\wedge y) + H_n(x\vee y)\,,
\]
where the last inequality follows from the convexity of $h$. 
Let $M>0$ be a constant such that $f\leq M$, $g \leq M$ and $h\geq -M$. Then, it holds
\[
F_n(x) + G_n(y) \leq \min (H_n(x\wedge y) ; 3M) + \min(H_n(x\vee y); 3M).  
\]

In other words $F_n$, $G_n$ and $H_n \wedge 3M$ satisfy the condition of the Four Functions Theorem (with $h_3=h_4$) so that, denoting by $m_n$ the uniform probability measure on $\Omega_n$,
\[
\int_{\Omega_n}  e^{F_n}\,dm_n  \int_{\Omega_n} e^{G_n}\,dm_n \leq \left( \int_{\Omega_n} e^{H_n\wedge 3M}\,dm_n \right)^2\,,
\]
Applying the Central Limit Theorem, one gets 
\[
\left(\int_{\mathbb{R}} e^{f} \,d\gamma \right)^{1/2} \left(\int_{\mathbb{R}} e^{g}\,\,d\gamma \right)^{1/2} \leq \int_{\mathbb{R}} e^{h\wedge 3M}\,d\gamma \leq \int_{\mathbb{R}} e^{h}\,d\gamma
\]
where $\gamma$ denotes the Standard Gaussian probability measure on $\R.$
Replacing $f,g,h$ by $f_\lambda(x):=f(\lambda^{1/2} x)$, $g_\lambda(x):=g(\lambda^{1/2}  x)$ and $h_\lambda(x):=h(\lambda^{1/2}  x)$, where $\lambda >0$, one easily gets
\[
\left(\int_{\mathbb{R}} e^{f(x)} e^{-\frac{x^2}{2\lambda}}\,dx \right)^{1/2} \left(\int_{\mathbb{R}} e^{g(y)}e^{-\frac{y^2}{2\lambda}}\,dy \right)^{1/2} \leq \int_{\mathbb{R}} e^{h(z)} e^{-\frac{z^2}{2\lambda}}\,dz.
\]
Letting $\lambda \to +\infty$, the monotone convergence theorem gives the desired inequality. Finally, one can easily remove the upper boundedness assumption on $f,g$ by truncation and monotone convergence. 
\end{proof}

\section{Klartag-Lehec Pr\'ekopa-Leindler inequality on $\Z$}\label{Sec2}
\subsection{From Klartag-Lehec Inequality to the Four Functions Theorem}
To make clear the connection with the preceding section, let us first remark that Theorem \ref{thm:KL} implies the one dimensional version of the Four Functions Theorem (and thus the result in all dimensions by tensorization).

Indeed let $f,g,h,k$ be four non-negative functions on $\{0,1\}$ satisfying the hypothesis of   the Four Functions Theorem, namely 
for any $x,y\in \{0,1\}$\,,
\[f(x)\, g(y)\leq h(x\wedge y)\, k(x\vee y).\]
Setting for any $x\in \Z$
\[\tilde f(x):= f(x)\1_{\{0,1\}}(x),\]
and similarly $\tilde g, \tilde h, \tilde k$,
one may easily check that that for any $x,y\in \Z$
\[\tilde f(x)\, \tilde g(y)\leq \tilde h\left(\left\lfloor \frac{x+y}{2}\right \rfloor\right)\tilde k\left(\left\lceil \frac{x+y}{2} \right\rceil\right).\]
Therefore applying Theorem \ref{thm:KL}  we get the conclusion of the
Four Functions Theorem, 
\[ (f(0)+f(1))(g(0)+g(1))\leq (h(0) +h(1))(k(0)+k(1)).\]

\subsection{Transport proof of the Klartag-Lehec Inequality}
Our goal is now to establish the following entropic version of Klartag-Lehec Inequality which is actually stronger than Theorem \ref{thm:KL}. In what follows, we recall that the \emph{monotone coupling} $\pi$ between two probability measures $\nu_0$ and $\nu_1$ on $\R$ is defined by
\[
\pi = \mathrm{Law} (F_{\nu_0}^{-1} (U) , F_{\nu_1}^{-1} (U) ),
\] 
where $U$ is a random variable uniformly distributed on $(0,1)$ and  where for all $i \in \{0,1\}$, $F_{\nu_i}(x) = \nu_i((-\infty,x])$, $x \in \R,$ is the cumulative distribution of $\nu_i$ and $F_{\nu_i}^{-1}(t) = \inf\{ x \in \R : F_{\nu_i}(x) \geq t\}$, $t \in (0,1)$, is the generalized inverse of $F_{\nu_i}.$
\begin{thm}[displacement convexity of entropy]\label{thm:KLentropic}
Suppose that $\nu_0,\nu_1$ are two probability measures on $\Z$ with compact supports. Define (recall the definition of the push forward right before Lemma \ref{transport})
\[
\nu_- = {m_-}_\# \pi \qquad \text{and}\qquad  \nu_+ = {m_+}_\# \pi,
\]
where $\pi$ is the monotone coupling between $\nu_0$ and $\nu_1$ , and for all $x,y\in \Z$,
\[m_-(x,y):=\left\lfloor \frac{x+y}{2}\right \rfloor,\qquad m_+(x,y):=\left\lceil \frac{x+y}{2} \right\rceil.\]
 Then, denoting by $m$ the counting measure on $\Z$, it holds
\begin{equation}\label{eq:dispconv}
H(\nu_-| m) + H(\nu_+| m) \leq H(\nu_{0}| m) + H(\nu_{1}| m) .
\end{equation}
\end{thm}
Before turning to the proof of Theorem \ref{thm:KLentropic}, let us first recall how to recover Theorem \ref{thm:KL} from Theorem \ref{thm:KLentropic}.
\proof[Proof of Theorem \ref{thm:KL}] The proof uses (again) the dual expression of the log-Laplace transform of any bounded function $\varphi$:     
\begin{eqnarray}\label{esp}
\log \int e^{\varphi} dm=\sup_{\nu}\left\{\int \varphi\, d\nu -H(\nu|m)\right\},
\end{eqnarray}
where the supremum runs over all probability measures $\nu$ on $\Z$ with bounded support.
Let $f,g,h,k$ be four non-negative functions satisfying \eqref{eq:hyp}. Given $\varepsilon, \kappa>0$ and setting $f^{\varepsilon,\kappa}(x)=\max(\varepsilon, \min(f(x),\kappa))$, one may simply check that   equivalently for all $x,y\in \Z$,
\begin{align*}
\log f^{\varepsilon,\kappa}(x)+\log g^{\varepsilon,\kappa}(y) 
& \leq 
\log h^{\varepsilon,\kappa}\left(m_-(x,y)\right)+\log k^{\varepsilon,\kappa}\left(m_+(x,y)\right).
\end{align*}
Integrating this inequality with respect to the monotone coupling $\pi$ of   two probability measures on $\Z$ with bounded support $\nu_0$ and $\nu_1$ implies
\begin{align*}
\int \log f^{\varepsilon,\kappa} \,d\nu_0+\int \log g^{\varepsilon,\kappa} \,d\nu_1
& \leq 
\int \log h^{\varepsilon,\kappa}(m_-) \,d\pi+\int \log k^{\varepsilon,\kappa}(m_+) \,d\pi \\
& = 
\int \log h^{\varepsilon,\kappa} \,d\nu_-+\int \log k^{\varepsilon,\kappa} \,d\nu_+. 
\end{align*}
Therefore, applying Inequality \eqref{eq:dispconv} of Theorem \ref{thm:KLentropic} implies
\begin{align*}
\int \log f^{\varepsilon,\kappa} \,d\nu_0 & - H(\nu_0|m)+\int \log g^{\varepsilon,\kappa} \,d\nu_1-H(\nu_1|m) \\
&\leq 
\int \log h^{\varepsilon,\kappa} \,d\nu_-- H(\nu_-| m) + \int \log k^{\varepsilon,\kappa} \,d\nu_+- H(\nu_+| m) \\
&\leq 
\log \int h^{\varepsilon,\kappa} \,dm+ \log \int k^{\varepsilon,\kappa} \,dm,
\end{align*}
where the last inequality is a consequence of  Identity \eqref{esp}. 
Then optimizing over all  probability measures with bounded support $\nu_0$ and $\nu_1$, and using again \eqref{esp} one gets 
\[\log \int f^{\varepsilon,\kappa} \,dm+ \log \int g^{\varepsilon,\kappa} \,dm \leq \log \int h^{\varepsilon,\kappa} \,dm+ \log \int k^{\varepsilon,\kappa} \,dm.\]
The conclusion of Theorem \ref{thm:KL} follows by monotone convergence as $\varepsilon$ goes to 0 and $\kappa$ goes to infinity.
\endproof

Now we turn to the proof of Theorem \ref{thm:KLentropic} which in turn is a consequence of the following result of independent interest.
\begin{thm}\label{thm:leq1}
With the same notation as in Theorem \ref{thm:KLentropic}, it holds
\begin{equation}\label{eq:leq1}
\sum_{(x,y) \in \Z^2} \frac{\nu_-(m_-(x,y)) \nu_{+}(m_+(x,y))}{\nu_0(x)\nu_1(y)} \pi(x,y) \leq 1.
\end{equation}
\end{thm}
\proof[Proof of Theorem \ref{thm:KLentropic}]
The logarithm function  being concave one gets by Jensen's inequality, thanks to \eqref{eq:leq1}, 
\[
H:=\sum_{(x,y) \in \Z^2} \log\left(\frac{\nu_-(m_-(x,y)) \nu_{+}(m_+(x,y))}{\nu_0(x)\nu_1(y)}\right) \pi(x,y) \leq 0.
\]
Now observe that, by definition of $\pi$, $\nu_-$ and $\nu_+$,
\begin{align*}
H &= \sum_{z \in \Z} \log (\nu_-(z))\nu_-(z) + \sum_{z \in \Z} \log (\nu_+(z))\nu_+(z) - \sum_{z \in \Z} \log (\nu_0(z))\nu_0(z) - \sum_{z \in \Z} \log (\nu_1(z))\nu_1(z) \\
& =H(\nu_-| m) + H(\nu_+| m) - H(\nu_{0}| m) - H(\nu_{1}| m)\,,
\end{align*}
completing the proof.
\endproof

In the proof of Theorem \ref{thm:leq1} we will make repeated use of the following elementary lemma:
\begin{lem}\ \label{lem:elem}
\begin{enumerate}
\item  Let $(x_1,y_1), (x_2,y_2) \in \Z^2$ be such that $(x_1,y_1) \neq (x_2,y_2)$ with $x_1 \leq x_2$ and $y_1 \leq y_2$. 
Then $\lfloor \frac{x_1+y_1}{2} \rfloor = \lfloor \frac{x_2+y_2}{2} \rfloor$ if and only if  $y_2-y_1 + x_2-x_1= 1$ and $\frac{x_1+y_1}{2} \in \Z$. \\
In this case, $\lceil \frac{x_2+y_2}{2} \rceil = \lceil \frac{x_1+y_1}{2} \rceil+1$.
\item Let $(x_1,y_1), (x_2,y_2) \in \Z^2$ be such that $x_1\leq x_2$, $y_1\leq y_2$, $\lfloor \frac{x_1+y_1}{2} \rfloor =a$ and $\lfloor \frac{x_2+y_2}{2} \rfloor = a'$ with $a < a'$.
\begin{itemize}
\item If $a' \geq a+2$, $\lceil \frac{x_1+y_1}{2} \rceil \neq \lceil \frac{x_2+y_2}{2} \rceil$.
\item If $a'=a+1$, $\lceil \frac{x_1+y_1}{2} \rceil = \lceil \frac{x_2+y_2}{2} \rceil$ if and only if $y_2-y_1 + x_2-x_1 =1$ with $ \frac{x_1 + y_1}{2} \in \Z + \frac{1}{2}$.
\end{itemize}
\end{enumerate}
\end{lem}
The following figures illustrate the next lemma.

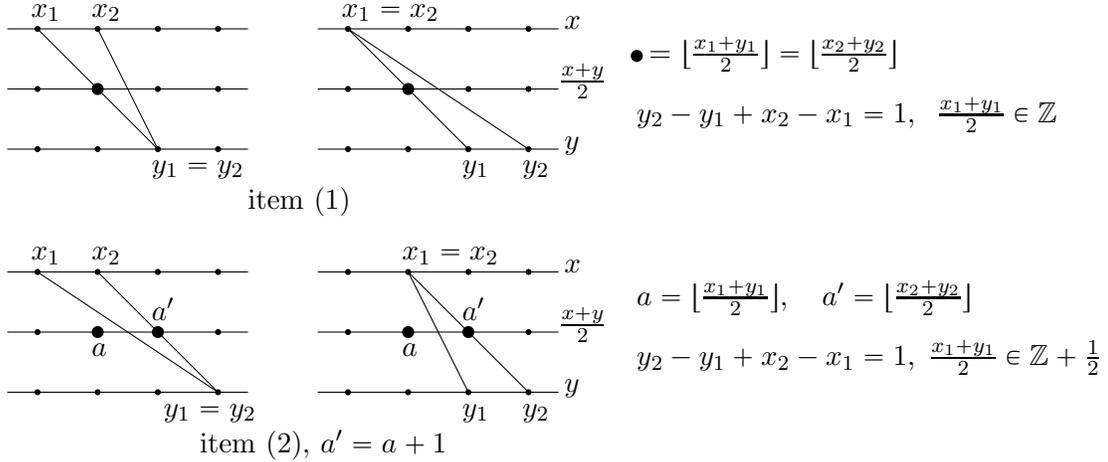
\setlength{\unitlength}{0,8cm}
\begin{picture}(5,4)
\put(0.5,1){\line(1,0){4}}
\put(0.5,2){\line(1,0){4}}
\put(0.5,3){\line(1,0){4}}
\put(3,1){\line(-2,2){2}}
\put(3,1){\line(-1,2){1}}
\put(1,1){\circle*{0.1}}
\put(2,1){\circle*{0.1}}
\put(3,1){\circle*{0.1}}
\put(4,1){\circle*{0.1}}
\put(1,2){\circle*{0.1}}
\put(2,2){\circle*{0.2}}
\put(3,2){\circle*{0.1}}
\put(4,2){\circle*{0.1}}
\put(1,3){\circle*{0.1}}
\put(2,3){\circle*{0.1}}
\put(3,3){\circle*{0.1}}
\put(4,3){\circle*{0.1}}
\put(0.9,3.2){$x_1$}
\put(1.9,3.2){$x_2$}
\put(2.9,0.6){$y_1=y_2$}
\put(4.5,0){item (1)}
\end{picture}
\begin{picture}(6,4)
\put(0.5,1){\line(1,0){4}}
\put(0.5,2){\line(1,0){4}}
\put(0.5,3){\line(1,0){4}}
\put(3,1){\line(-2,2){2}}
\put(4,1){\line(-3,2){3}}
\put(1,1){\circle*{0.1}}
\put(2,1){\circle*{0.1}}
\put(3,1){\circle*{0.1}}
\put(4,1){\circle*{0.1}}
\put(1,2){\circle*{0.1}}
\put(2,2){\circle*{0.2}}
\put(5.8,2.55){\circle*{0.2}}
\put(6.0,2.45){= $\lfloor\frac{x_1+y_1}{2} \rfloor = \lfloor \frac{x_2+y_2}{2} \rfloor$ }
\put(3,2){\circle*{0.1}}
\put(5.8,1.45){$y_2-y_1 + x_2-x_1= 1, \;\; \frac{x_1+y_1}{2}\in \Z$}
\put(4,2){\circle*{0.1}}
\put(1,3){\circle*{0.1}}
\put(2,3){\circle*{0.1}}
\put(3,3){\circle*{0.1}}
\put(4,3){\circle*{0.1}}
\put(4.6,3){$x$}
\put(4.5,2){$\frac{x+y}2$}
\put(4.6,1){$y$}
\put(0.9,3.2){$x_1=x_2$}
\put(2.9,0.6){$y_1$}
\put(3.9,0.6){$y_2$}
\end{picture}

\begin{picture}(5,4)
\put(0.5,1){\line(1,0){4}}
\put(0.5,2){\line(1,0){4}}
\put(0.5,3){\line(1,0){4}}
\put(4,1){\line(-3,2){3}}
\put(4,1){\line(-2,2){2}}
\put(1,1){\circle*{0.1}}
\put(2,1){\circle*{0.1}}
\put(3,1){\circle*{0.1}}
\put(4,1){\circle*{0.1}}
\put(1,2){\circle*{0.1}}
\put(2,2){\circle*{0.2}}
\put(1.9,1.6){$a$}
\put(3,2){\circle*{0.2}}
\put(2.9,2.2){$a'$}
\put(4,2){\circle*{0.1}}
\put(1,3){\circle*{0.1}}
\put(2,3){\circle*{0.1}}
\put(3,3){\circle*{0.1}}
\put(4,3){\circle*{0.1}}
\put(0.9,3.2){$x_1$}
\put(1.9,3.2){$x_2$}
\put(3.1,0.6){$y_1=y_2$}
\put(3.7,0){item (2), $ a'=a+1$}
\end{picture}
\begin{picture}(6,4)
\put(0.5,1){\line(1,0){4}}
\put(0.5,2){\line(1,0){4}}
\put(0.5,3){\line(1,0){4}}
\put(3,1){\line(-1,2){1}}
\put(4,1){\line(-2,2){2}}
\put(1,1){\circle*{0.1}}
\put(2,1){\circle*{0.1}}
\put(3,1){\circle*{0.1}}
\put(4,1){\circle*{0.1}}
\put(1,2){\circle*{0.1}}
\put(2,2){\circle*{0.2}}
\put(1.9,1.6){$a$}
\put(5.8,2.45){$a=\lfloor \frac{x_1+y_1}{2} \rfloor ,\quad a'=\lfloor \frac{x_2+y_2}{2} \rfloor $}
\put(3,2){\circle*{0.2}}
\put(2.9,2.2){$a'$}
\put(5.8,1.45){$y_2-y_1 + x_2-x_1= 1,\; \frac{x_1 + y_1}{2} \in \Z + \frac{1}{2}$}
\put(4,2){\circle*{0.1}}
\put(1,3){\circle*{0.1}}
\put(2,3){\circle*{0.1}}
\put(3,3){\circle*{0.1}}
\put(4,3){\circle*{0.1}}
\put(4.6,3){$x$}
\put(4.5,2){$\frac{x+y}2$}
\put(4.6,1){$y$}
\put(1.9,3.2){$x_1=x_2$}
\put(2.9,0.6){$y_1$}
\put(3.9,0.6){$y_2$}
\end{picture}

\proof[Proof of Lemma \ref{lem:elem}]
(1) If $y_2-y_1 + x_2 - x_1 \geq 2$, then $\frac{x_2+y_2}{2} \geq \frac{x_1+y_1}{2} + 1$ and thus $\lfloor \frac{x_2+y_2}{2} \rfloor \geq  \lfloor \frac{x_1+y_1}{2} \rfloor + 1.$ Hence 
$y_2-y_1 + x_2 - x_1 =1.$ Without loss of generality one can assume that $x_1=x_2$ and $y_2 = y_1+1$. But in this case, $\frac{x_2+y_2}{2} = \frac{x_1+y_1}{2} + \frac{1}{2}$. The fact that $\lfloor \frac{x_1+y_1}{2} \rfloor = \lfloor \frac{x_2+y_2}{2} \rfloor$ then implies that $ \frac{x_1+y_1}{2} \in \Z.$ The converse is obvious. In this case $\lceil \frac{x_2+y_2}{2} \rceil = \lceil \frac{x_1+y_1}{2} + \frac{1}{2} \rceil = \frac{x_1+y_1}{2} +1 =  \lceil \frac{x_1+y_1}{2} \rceil +1$.

(2) If $a' \geq a+2$, then 
\[
\lceil \frac{x_2+y_2}{2} \rceil  \geq \lfloor \frac{x_2+y_2}{2} \rfloor = a' \geq a+2 = \lfloor \frac{x_1+y_1}{2} \rfloor +2 \geq \lceil \frac{x_1+y_1}{2} \rceil +1.
\]
Now let us assume that $a'=a+1$. If $y_2-y_1 + x_2-x_1= 2$, then $\frac{x_2+y_2}{2} = \frac{x_1+y_1}{2}+1$ and so $ \lceil \frac{x_2+y_2}{2} \rceil=\lceil \frac{x_1+y_1}{2} \rceil+1.$ Therefore $y_2-y_1 + x_2-x_1= 1$ and so $\frac{x_2+y_2}{2} = \frac{x_1+y_1}{2}+\frac{1}{2}$. The condition $\lfloor \frac{x_2+y_2}{2} \rfloor = \lfloor \frac{x_1+y_1}{2} \rfloor +1$ then implies that $ \frac{x_1 + y_1}{2} \in \Z + \frac{1}{2}$. Then it holds $\lceil \frac{x_2+y_2}{2} \rceil = \lceil \frac{x_1+y_1}{2} + \frac{1}{2} \rceil = \frac{x_1+y_1}{2} + \frac{1}{2} = \lceil \frac{x_1+y_1}{2} \rceil.$ The converse is obvious.
\endproof

Before proving Theorem \ref{thm:leq1} let us introduce some notation. We will denote 
\[
M_- = \{m_-(x,y) : (x,y) \in \mathrm{supp}(\pi)\}
\]
and for all $a \in \Z$, 
\[
S(a) = \{(x,y) \in \mathrm{supp}(\pi) : m_-(x,y) = a\}
\]
(with thus $S(a)=\emptyset$ when $a \notin M_-$).
\begin{lem}\label{lem:card}
For any $a\in \Z$, $\mathrm{Card}(S(a))\in \{0,1,2\}$. 
\end{lem}
\proof[Proof of Lemma \ref{lem:card}]
Let $a \in M_-$. By compactness of the support of $\pi$, the set $S(a)$ is finite. 
Suppose that $\mathrm{Card} (S(a))>1$. Let $x_0$ be the minimal first coordinate of the elements of $S(a)$, and let $y_0$ be the minimal second coordinate of the elements of $S(a)$ having $x_0$ as first coordinate. If $(x_1,y_1)$ is another element of $S(a)$, then either $x_0=x_1$ and $y_0\leq y_1$, or $x_0 < x_1$ and in this case, by monotonicity of the support of $\pi$, one has $y_0 \leq y_1$. According to Item (1) of Lemma \ref{lem:elem}, one has $\frac{x_0+y_0}{2} \in \Z$ and ($x_0=x_1$ and $y_1=y_0+1$) or ($y_0=y_1$ and $x_1=x_0+1$). By monotonicity of the support of $\pi$, these two cases exclude each other and so $\mathrm{Card}(S(a))=2.$
\endproof
For $i \in \{1,2\}$, we will denote by $M_-^{i}$ the set of $a \in M_-$ such that $\mathrm{Card}(S(a))=i$.
If $a \in M_-^{1}$, the unique element of $S(a)$ will be denoted by $(x_0(a),y_0(a))$. If $a \in M_-^{2}$, we will denote by $(x_0(a),y_0(a))$ and $(x_1(a),y_1(a))$ the two elements of $S(a)$, with the convention that $x_0(a)\leq x_1(a)$ and $y_0(a)\leq y_1(a)$ and $\frac{x_0(a)+y_0(a)}{2} \in \Z$ as in Lemma \ref{lem:elem} and the proof above.

\proof[Proof of Theorem~\ref{thm:leq1}] Using the notation above, we need to show that the following quantity is less than or equal to $1$.
\begin{align*}
P&:=\sum_{(x,y) \in \Z^2} \!\! \frac{\nu_-(m_-(x,y)) \nu_{+}(m_+(x,y))}{\nu_0(x)\nu_1(y)} \pi(x,y) 
= \sum_{a \in M_-} \sum_{(x,y) \in S(a)} \!\!  \frac{\nu_-(a) \nu_{+}(m_+(x,y))}{\nu_0(x)\nu_1(y)}\pi(x,y) .
\end{align*}
The strategy to bound $P$ by 1 is to show that, in fact,
\begin{eqnarray}\label{Espagna}
P \leq \sum_{a \in M_-} \sum_{(x,y) \in S(a)} \pi(x,y) = 1.
\end{eqnarray}
For  that purpose we consider two cases.

\noindent \textbf{First case.} Let $a \in M_-$ be such that 
\begin{equation}\label{eq:non-inter}
m_+(S(a))\cap m_+(S(a-1)) =\emptyset  \qquad \text{and}\qquad m_+(S(a))\cap m_+(S(a+1)) =\emptyset .
\end{equation}
Then let us show that for all $(x,y) \in S(a)$, it holds 
\begin{equation}\label{eq:case1}
\frac{\nu_-(a) \nu_{+}(m_+(x,y))}{\nu_0(x)\nu_1(y)} \leq 1.
\end{equation}
We distinguish between two sub-cases, $a \in M_-^1$ and $a \in M_-^2$.
Suppose first that $a \in M_-^1$. Then $S(a) = \{(x_0,y_0)\}$ and therefore
\[
\nu_-(a) = \pi(\{(u,v) : m_-(u,v) = a\}) = \pi(x_0,y_0).
\]
Moreover, since $a$ satisfies \eqref{eq:non-inter}, Item 2 of  Lemma \ref{lem:elem} gives that
\[
\nu_+(m_+(x_0, y_0)) = \pi(\{(u,v) \in \Z^2 : m_+(u,v) = m_+(x_0,y_0)\}) = \pi((x_0,y_0)).
\]
Since $\pi(x_0,y_0) \leq \min(\nu_0(x_0), \nu_1(y_0))$, this gives \eqref{eq:case1}.
Now let us assume that $a \in M_-^2$. Then one can assume without loss of generality that $S(a) = \{(x_0,y_0),  (x_0,y_0+1)\}$ with $\frac{x_0+y_0}{2} \in \Z$ and thus $m_+(x_0,y_0+1) = m_+(x_0,y_0)+1.$
In this case,
\[
\nu_-(a) = \pi(\{(u,v) : m_-(u,v) = a\}) = \pi(x_0,y_0) + \pi(x_0, y_0+1) \leq \nu_0(x_0)
\]
and reasoning as above
\[
\nu_+(m_+(x_0,y_0)) = \pi(x_0,y_0) \leq \nu_1(y_0) \ \  \text{and} \ \  \nu_+(m_+(x_0,y_0+1)) = \pi(x_0,y_0+1) \leq \nu_1(y_0+1)\,,
\]
which establish \eqref{eq:case1}.

\bigskip

\noindent \textbf{Second case.} Let $a_0 \in M_-$ and $p\geq1$ such that $m_+(S(a_0+i)) \cap m_+(S(a_0+i+1)) \neq \emptyset$ for all $i \in \{0,\ldots,p-1\}$ and such that $m_+(S(a_0-1)) \cap m_+(S(a_0))= \emptyset$ and $m_+(S(a_0+p)) \cap m_+(S(a_0+p+1)) = \emptyset$ (\textit{i.e.}\ $p$ is maximal). Since $m_+(S(a_0+i)) \subset \{a_0+i ; a_0+i+1\}$, the only possibility is that $m_+(S(a_0+i)) = \{a_0+i ; a_0+i+1\}$ for all $i \in \{1,\ldots,p-1\}$ (this set being empty if $p=1$). 
Let us assume that $a_0 \in M_-^2$ and $a_0+p\in M_-^2$ (the other cases are dealt similarly).
Let us denote $(x_0^i, y_0^i) = (x_0(a_0+i), y_0 (a_0+i))$ and $(x_1^i, y_1^i) = (x_1(a_0+i), y_1 (a_0+i))$ (recall that by definition $x_1^i \geq x_0^i$ and $y_1^i \geq y_0^i$). According to Lemma \ref{lem:elem}, it holds $x_1^{i}- x_0^i + y_1^i - y_0^i =1$ and $y_0^{i+1}- y_1 ^i + x_0^{i+1} - x_1 ^i = 1$.

Let us introduce
\begin{align*}
P_{a_0} &:= \sum_{i=0}^p \sum_{(x,y) \in S(a_0+i)} \frac{\nu_-(a_0+i) \nu_{+}(m_+(x,y))}{\nu_0(x)\nu_1(y)} \pi(x,y)\\
& =  \sum_{i=0}^p \left[  \frac{\nu_-(a_0+i) \nu_{+}(a_0+i)}{\nu_0(x_0^i)\nu_1(y_0 ^i)} \pi(x_0 ^i,y_0^i)+  \frac{\nu_-(a_0+i) \nu_{+}(a_0+i+1)}{\nu_0(x_1 ^i)\nu_1(y_1^i)} \pi(x_1^i,y_1^i)\right]\\
\end{align*}
and let us show that 
\begin{equation}\label{eq:boundP}
P_{a_0} \leq \sum_{i=0}^p\sum_{(x,y) \in S(a_0+i)} \pi(x,y) .
\end{equation}
We will use the following facts:
\begin{itemize}
\item Fact 1 : For all $i \in \{0,\ldots,p\}$ it holds
\[
\nu_-(a_0+i) = \pi(x_0^i,y_0^i) + \pi(x_1^i,y_1^i)\,.
\]
\item Fact 2 : For all $i \in \{1,\ldots, p\}$
\[
\nu_+(a_o+i) = \pi(x_1^{i-1},y_1^{i-1})+\pi(x_0^i,y_0^i) 
\]
and $\nu_+(a_0) = \pi(x_0^0, y_0^0)$ and $\nu_+(a_0+p+1) = \pi(x_1^p,y_1^p)$.
\end{itemize}
Observe that
\[
P_{a_0} = \sum_{i=0}^{p+1} \alpha_i \nu_+(a_0+i),
\]
where, for $i\in \{1,\ldots,p\}$,
\[
\alpha_i =  \frac{\nu_-(a_0+i-1)}{\nu_0(x_1^{i-1}) \nu_1(y_1^{i-1})} \pi(x_1^{i-1},y_1^{i-1})+ \frac{\nu_-(a_0+i)}{\nu_0(x_0^i) \nu_1(y_0^i)} \pi(x_0^i,y_0^i).
\]
and 
\[
\alpha_0 = \frac{\nu_-(a_0)}{\nu_0(x_0^0) \nu_1(y_0^0)} \pi(x_0^0,y_0^0)\qquad \alpha_{p+1} = \frac{\nu_-(a_0+p)}{\nu_0(x_1^{p}) \nu_1(y_1^{p})} \pi(x_1^{p},y_1^{p}).
\]
According to Fact 2, in order to prove \eqref{eq:boundP}, it is enough to show that $\alpha_i \leq 1$ for all $i\in \{0,\ldots,p+1\}.$

For $i = p+1$, one can assume without loss of generality that $x_0^p = x_1^p$. Then, according to Fact 1, $\nu_-(a_0+p) \leq  \nu_0(x_1^p)$ and since $\pi(x_1^p,y_1^p) \leq \nu_1(y_1^p)$, it follows that $\alpha_{p+1} \leq 1$. The case $i=0$ is similar.

Now let us consider the case $i \in \{1,\ldots,p\}$. Observe that either $x_1^{i-1} = x_0^i$ either $y_1^{i-1} = y_0^i$. 
Without loss of generality, one can assume that $x_1^{i-1} = x_0^i$ (the  case $y_1^{i-1} = y_0^i$  follows by symmetry in $x$ and $y$), so that 
\[
\alpha_i =  \frac1{\nu_0(x_1^{i-1}) } \left(\frac{\nu_-(a_0+i-1)\, \pi(x_1^{i-1},y_1^{i-1})}{ \nu_1(y_1^{i-1})} + \frac{\nu_-(a_0+i)\, \pi(x_0^i,y_0^i)}{ \nu_1(y_0^i)} \right).
\]

Let us consider the following subcases:
\begin{itemize}
\item[(a)] If $x_0^{i-1} = x_1^{i-1} = x_0^i = x_1^i$, then 
$y_1^{i-1}, y_0^i, y_1^i$ are pairwise distinct. 
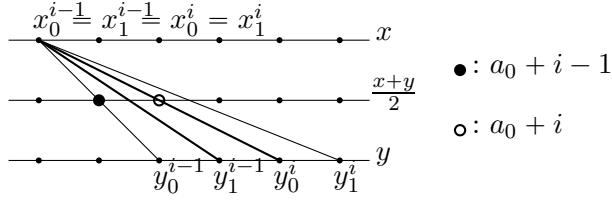
\begin{center}
\setlength{\unitlength}{0,8cm}
\begin{picture}(8,4)
\put(0.5,1){\line(1,0){6}}
\put(0.5,2){\line(1,0){6}}
\put(0.5,3){\line(1,0){6}}
\put(6,1){\line(-5,2){5}}
\put(3,1){\line(-2,2){2}}
\thicklines
\put(5,1){\line(-4,2){4}}
\put(4,1){\line(-3,2){3}}
\put(1,1){\circle*{0.1}}
\put(2,1){\circle*{0.1}}
\put(3,1){\circle*{0.1}}
\put(4,1){\circle*{0.1}}
\put(5,1){\circle*{0.1}}
\put(6,1){\circle*{0.1}}
\put(1,2){\circle*{0.1}}
\put(2,2){\circle*{0.2}}
\put(8,2.5){\circle*{0.2}}
\put(8.2,2.45){:$\,\,a_0+i-1$ }
\put(3,2){\circle{0.17}}
\put(8,1.5){\circle{0.17}}
\put(8.2,1.45){:$\,\,a_0+i$ }
\put(4,2){\circle*{0.1}}
\put(5,2){\circle*{0.1}}
\put(6,2){\circle*{0.1}}
\put(1,3){\circle*{0.1}}
\put(2,3){\circle*{0.1}}
\put(3,3){\circle*{0.1}}
\put(4,3){\circle*{0.1}}
\put(5,3){\circle*{0.1}}
\put(6,3){\circle*{0.1}}
\put(6.6,3){$x$}
\put(6.5,2){$\frac{x+y}2$}
\put(6.6,1){$y$}
\put(0.9,3.2){$x_0^{i-1}\!\!\!\!\!=x_1^{i-1}\!\!\!\!\!=x_0^i=x_1^i$}
\put(2.9,0.6){$y_0^{i-1}$}
\put(3.9,0.6){$y_1^{i-1}$}
\put(4.9,0.6){$y_0^{i}$}
\put(5.9,0.6){$y_1^{i}$}
\end{picture}
\end{center}
Since   $\pi(x_1^{i-1},y_1^{i-1}) \leq \nu_1(y_1^{i-1})$ and $ \pi(x_0^i,y_0^i) \leq \nu_1(y_0 ^i)$, by  using Fact 1, one gets 
\[
\alpha_i \leq \frac{\nu_-(a_0+i-1)+\nu_-(a_0+i)}{\nu_0(x_1^{i-1})}  = \frac{\pi(x_0^{i-1},y_0^{i-1}) + \pi(x_1^{i-1},y_1^{i-1}) +\pi(x_0^i,y_0^i) + \pi(x_1^i,y_1^i) }{\nu_0(x_1^{i-1})} \leq 1.
\]
The last inequality holds since $x_0^{i-1} = x_1^{i-1} = x_0^i = x_1^i$ and $y_1^{i-1}, y_0^i, y_1^i$ are pairwise distinct. 
\item[(b)] If $x_0^{i-1} \neq  x_1^{i-1} = x_0^i = x_1^i$, then necessarily $y_0^{i-1} = y_1^{i-1}$. 
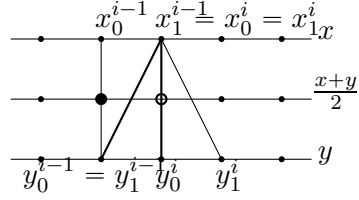
\begin{center}
\begin{picture}(6,4)
\put(0.5,1){\line(1,0){5}}
\put(0.5,2){\line(1,0){5}}
\put(0.5,3){\line(1,0){5}}
\put(2,3){\line(0,-1){2}}
\put(3,3){\line(1,-2){1}}
\thicklines
\put(2,1){\line(1,2){1}}
\put(3,3){\line(0,-1){2}}
\put(1,1){\circle*{0.1}}
\put(2,1){\circle*{0.1}}
\put(3,1){\circle*{0.1}}
\put(4,1){\circle*{0.1}}
\put(5,1){\circle*{0.1}}
\put(1,2){\circle*{0.1}}
\put(2,2){\circle*{0.2}}
\put(3,2){\circle{0.17}}
\put(4,2){\circle*{0.1}}
\put(5,2){\circle*{0.1}}
\put(1,3){\circle*{0.1}}
\put(2,3){\circle*{0.1}}
\put(3,3){\circle*{0.1}}
\put(4,3){\circle*{0.1}}
\put(5,3){\circle*{0.1}}

\put(5.6,3){$x$}
\put(5.5,2){$\frac{x+y}2$}
\put(5.6,1){$y$}
\put(1.9,3.2){$x_0^{i-1}$}
\put(2.9,3.2){$x_1^{i-1}\!\!\!\!\!=x_0^i=x_1^i$}
\put(0.7,0.6){$y_0^{i-1}=y_1^{i-1}$}
\put(2.9,0.6){$y_0^{i}$}
\put(3.9,0.6){$y_1^{i}$}
\end{picture}
\end{center}
Using Fact 1, one gets $\nu_-(a_0+i-1) \leq \nu_1(y_1^{i-1})$. 
Since $\nu_-(a_0+i) = \pi(x_0^i,y_0^i) + \pi(x_1^i,y_1^i)$ and $\pi(x_0^i,y_0^i)\leq \nu_1(y_0^i)$ one gets
\[
\alpha_i \leq \frac{1}{\nu_0(x_1^{i-1})} \left[\pi(x_1^{i-1},y_1^{i-1}) + \pi(x_0^i,y_0^i)+ \pi(x_1^i, y_1^i)\right] \leq 1.
\]
\item[(c)] If $x_0^{i-1} =  x_1^{i-1} = x_0^i \neq x_1^i$, then necessarily $y_0^{i} = y_1^{i}$. 
\begin{center}
\begin{picture}(6,4)
\put(0.5,1){\line(1,0){5}}
\put(0.5,2){\line(1,0){5}}
\put(0.5,3){\line(1,0){5}}
\put(2,3){\line(0,-1){2}}
\put(3,3){\line(1,-2){1}}
\thicklines
\put(2,3){\line(1,-2){1}}
\put(2,3){\line(1,-1){2}}
\put(1,1){\circle*{0.1}}
\put(2,1){\circle*{0.1}}
\put(3,1){\circle*{0.1}}
\put(4,1){\circle*{0.1}}
\put(5,1){\circle*{0.1}}
\put(1,2){\circle*{0.1}}
\put(2,2){\circle*{0.2}}
\put(3,2){\circle{0.17}}
\put(4,2){\circle*{0.1}}
\put(5,2){\circle*{0.1}}
\put(1,3){\circle*{0.1}}
\put(2,3){\circle*{0.1}}
\put(3,3){\circle*{0.1}}
\put(4,3){\circle*{0.1}}
\put(5,3){\circle*{0.1}}
\put(5.6,3){$x$}
\put(5.5,2){$\frac{x+y}2$}
\put(5.6,1){$y$}
\put(0,3.2){$x_0^{i-1}\!\!\!\!\!=x_1^{i-1}\!\!\!\!\!=x_0^i$}
\put(2.9,3.2){$x_1^i$}
\put(1.9,0.6){$y_0^{i-1}$}
\put(2.9,0.6){$y_1^{i-1}$}
\put(3.9,0.6){$y_0^{i}=y_1^{i}$}
\end{picture}
\end{center}
Using Fact 1, one gets $\nu_-(a_0+i) \leq \nu_1(y_0^{i})$. Since $\nu_-(a_0+i-1) = \pi(x_0^{i-1},y_0^{i-1}) + \pi(x_1^{i-1},y_1^{i-1})$ and $\pi(x_1^{i-1},y_1^{i-1})\leq \nu_1(y_0^{i-1})$,  it follows that 
\[
\alpha_i \leq \frac{1}{\nu_0(x_1^{i-1})} \left[\pi(x_0^{i-1},y_0^{i-1}) + \pi(x_1^{i-1},y_1^{i-1})+ \pi(x_0^i, y_0^i)\right] \leq 1.
\]

\item[(d)] If $x_0^{i-1} \neq  x_1^{i-1}$, $x_1^{i-1} = x_0^i$ $x_0^i\neq x_1^i$, then necessarily $y_0^{i-1} =  y_1^{i-1}$ and $y_0^i= y_1^i$.
\begin{center}
\begin{picture}(6,4)
\put(0.5,1){\line(1,0){5}}
\put(0.5,2){\line(1,0){5}}
\put(0.5,3){\line(1,0){5}}
\put(4,1){\line(-1,2){1}}
\put(3,1){\line(-2,2){2}}
\thicklines
\put(4,1){\line(-2,2){2}}
\put(3,1){\line(-1,2){1}}
\put(1,1){\circle*{0.1}}
\put(2,1){\circle*{0.1}}
\put(3,1){\circle*{0.1}}
\put(4,1){\circle*{0.1}}
\put(5,1){\circle*{0.1}}
\put(1,2){\circle*{0.1}}
\put(2,2){\circle*{0.2}}
\put(3,2){\circle{0.17}}
\put(4,2){\circle*{0.1}}
\put(5,2){\circle*{0.1}}
\put(1,3){\circle*{0.1}}
\put(2,3){\circle*{0.1}}
\put(3,3){\circle*{0.1}}
\put(4,3){\circle*{0.1}}
\put(5,3){\circle*{0.1}}
\put(5.6,3){$x$}
\put(5.5,2){$\frac{x+y}2$}
\put(5.6,1){$y$}
\put(0.8,3.2){$x_0^{i-1}$}
\put(1.6,3.2){$x_1^{i-1}\!\!\!\!\!\!\!=x_0^i$}
\put(3,3.2){$x_1^i$}
\put(2,0.6){$y_0^{i-1}\!\!\!\!\!=y_1^{i-1}$}
\put(3.9,0.6){$y_0^{i}=y_1^{i}$}
\end{picture}
\end{center}
Reasoning as in the preceding cases, one gets $\nu_-(a_0+i-1) \leq \nu_1(y_1^{i-1})$, $\nu_-(a_0+i)\leq \nu_1(y_1^i)$ and so 
\[
\alpha_i \leq \frac{1}{\nu_0(x_1^{i-1})} \left[\pi(x_1^{i-1},y_1^{i-1}) + \pi(x_0^i,y_0^i)\right] \leq 1.
\]
\end{itemize}

Conclusion : by considering successively the elements $a \in M_-$ in increasing order, case 1 can be  repeated successively several times and we may pass from case 1 to case 2 or from case 2 to case 1. Therefore  
after a finite use of cases 1 and 2 described above,  \eqref{eq:case1} and \eqref{eq:boundP} imply \eqref{Espagna}. This concludes the proof of Theorem~\ref{thm:leq1}.
\endproof

\subsection{From the Klartag-Lehec Inequality to the Pr\'ekopa-Leindler Inequality}

First, let us explain how to recover the conclusion of Theorem \ref{thm:PL} for $t=1/2$ and continuous functions using Theorem \ref{thm:KL}.
More precisely we are going to show that if $F,G,H, K:\R \to \R^+$ are continuous functions such that
\[
F(x)G(y) \leq H\left(\frac{x+y}{2}\right)K\left(\frac{x+y}{2}\right),\qquad \forall x,y \in \R
\]
then
\begin{equation}\label{eq:PL}
\int F(x)\,dx \int G(x)\,dx \leq \int H(x)\,dx \int K(x)\,dx.
\end{equation}
Then taking in particular $H=K$ gives the conclusion of Theorem \ref{thm:PL} for $t=1/2$.
\proof[Proof of \eqref{eq:PL}]
Let $N\geq 1$ and for all positive integer $n$ consider the grid $x_i^n = -N + 2\frac{iN}{n} $, $i\in \{0,\ldots,n\}$.
Define $f,g,h,k : \Z \to \R^+$ as follows : 
$$
f(i):=\begin{cases}
F(x_i^n) & \mbox{if } i\in \{0,\ldots, n\} \\
 0 & \mbox{otherwise}
\end{cases} ,
\;
h(i):=\begin{cases}
\max( H(x_i^n) , H(x_i^n + \frac{N}{n})) & \mbox{if } i\in \{0,\ldots, n\} \\
 0 & \mbox{otherwise}
\end{cases} ,
$$
$$
g(i):=\begin{cases}
G(x_i^n) & \mbox{if } i\in \{0,\ldots, n\} \\
 0 & \mbox{otherwise}
\end{cases} ,
\;
k(i):=\begin{cases}
\max( K(x_i^n) , K(x_i^n - \frac{N}{n})) & \mbox{if } i\in \{0,\ldots, n\} \\
 0 & \mbox{otherwise}
\end{cases} .
$$
If $i,j \in \{0,\ldots, n\}$ then, there is some $\ep \in \{0,1\}$ such that
\[
\frac{x_i^n+x_j^n}{2} = -N + 2\frac{ \lfloor \frac{i+j}{2}\rfloor N}{n} + \ep \frac{N}{n} = x^n_{ \lfloor \frac{i+j}{2}\rfloor} +  \ep \frac{N}{n}
\]
and so $H(\frac{x_i^n+x_j^n}{2}) \leq h(\lfloor \frac{i+j}{2}\rfloor)$. Similarly, $K(\frac{x_i^n+x_j^n}{2}) \leq k(\lceil \frac{i+j}{2}\rceil)$.
Therefore, for all $i, j \in \{0,\ldots,n\}$,
\[
f(i)g(i) = F(x_i^n)G(x_i^n) \leq H\left(\frac{x_i^n+x_j^n}{2}\right)K\left(\frac{x_i^n+x_j^n}{2}\right) \leq h\left(\left\lfloor \frac{i+j}{2}\right\rfloor\right)k\left(\left\lceil \frac{i+j}{2}\right\rceil\right)\,. 
\]
The functions $f,g,h,k$ thus satisfy the assumption of Theorem \ref{thm:KL} and so 
\[
\left[\sum_{i=0}^n F(x_i^n)\right] \left[\sum_{i=0}^n G(x_i^n)\right]  \leq \left[\sum_{i=0}^n \max (H(x_i^n), H(x_i^n +\frac{N}{n})\right] \left[\sum_{i=0}^n \max(K(x_i^n) ; K(x_i^n - \frac{N}{n})\right] .
\]
By uniform continuity of $f,g,h,k$ on $[-2N,2N]$, multiplying both sides by $(2N/n)^2$ and letting $n \to +\infty$, it follows that
\[
\int_{-N}^N F(x)\,dx \int_{-N}^N G(x)\,dx \leq \int_{-N}^N H(x)\,dx \int_{-N}^N K(x)\,dx.
\]
Finally, letting $N \to +\infty$ gives \eqref{eq:PL}.
\endproof
\subsection{Displacement convexity of entropy : from discrete to continuous}
In the same vein as in the previous sub-section, one can deduce from Theorem \ref{thm:KLentropic}, the following well-known continuous version of the displacement convexity of the relative entropy with respect to Lebesgue measure.
\begin{thm}
Let $\nu_0,\nu_1$ be probability measures on $\R$ with compact supports and define $\nu_{1/2}$ as the law of $\frac{X_0+X_1}{2}$, where $(X_0,X_1)$ is distributed according to the monotone rearrangement coupling $\pi$ between $\nu_0$ and $\nu_1$. Then it holds
\[
2 H(\nu_{1/2} | \mathrm{Leb}) \leq H(\nu_{0} | \mathrm{Leb})+ H(\nu_{1} | \mathrm{Leb}).
\]
\end{thm}
\proof
Without loss of generality, one can assume that $ H(\nu_{0} | \mathrm{Leb})+ H(\nu_{1} | \mathrm{Leb}) <+\infty$. Consider $(X_0,X_1)$ distributed according to $\pi$ and define, for $n\geq 1$, $\pi^n = \mathrm{Law} \left( \frac{\lfloor nX_0\rfloor}{n} , \frac{\lfloor nX_1\rfloor}{n} \right)$ and $\nu_0^n = \mathrm{Law} \left( \frac{\lfloor nX_0\rfloor}{n} \right)$ and $\nu_1^n = \mathrm{Law} \left( \frac{\lfloor nX_1\rfloor}{n} \right)$. The coupling $\pi^n$ is easily seen to be monotone. Since Theorem \ref{thm:KLentropic} immediately extends to probability measures on $\frac{1}{n}\Z$, one gets 
\begin{equation}\label{eq:Disp-1}
H(\nu_-^n | m^n) + H(\nu_+^n | m^n) \leq H(\nu_0^n | m^n) + H(\nu_1^n | m^n),
\end{equation}
where $m^n$ is the counting measure on $\frac{1}{n}\Z$ and 
\[
\nu_-^n = \mathrm{Law} \left(\frac{1}{n}\left\lfloor \frac{\lfloor nX_0 \rfloor + \lfloor nX_1 \rfloor }{2}\right\rfloor \right) \qquad\text{and}\qquad \nu_+^n = \mathrm{Law} \left(\frac{1}{n}\left\lceil \frac{\lfloor nX_0 \rfloor + \lfloor nX_1 \rfloor }{2}\right\rceil \right).
\]
Assuming that $\nu_0([-K,K[) = \nu_1([-K,K[) =1$, where $K\geq 1$ is an integer and denoting by $\mu^n$ the probability measure $\frac{1}{2nK}\mathbf{1}_{[-K,K[} m^n$, \eqref{eq:Disp-1} is equivalent to
\begin{equation}\label{eq:Disp-2}
H(\nu_-^n | \mu^n) + H(\nu_+^n | \mu^n) \leq H(\nu_0^n | \mu^n) + H(\nu_1^n | \mu^n).
\end{equation}
Let $\mu$ be the uniform (continuous) distribution on $[-K,K[.$ On the one hand, for $i \in \{0,1\}$
\begin{align*}
H(\nu_i^n | \mu^n) &= \sum_{k = -nK}^{nK-1} \nu_0^n\left(\frac{k}{n}\right)\log \left(\frac{\nu_i^n(\frac{k}{n})}{\mu^n(\frac{k}{n})}\right) \\
&= \sum_{k=-nK}^{nK-1} \P\left(X_i \in \left[\frac{k}{n}, \frac{k+1}{n}\right[ \right) \log \left(\frac{\P(X_i \in [\frac{k}{n}, \frac{k+1}{n}[ )}{\mu([\frac{k}{n}, \frac{k+1}{n}[)}\right) \\
& \leq  \sum_{k=-nK}^{nK-1}  \int_{\frac{k}{n}}^{\frac{k+1}{n}} \log\left(\frac{d\nu_i}{d\mu}\right) d\nu_i = H(\nu_i |\mu),
\end{align*}
where the inequality comes from Jensen's inequality applied to the convex function $x \mapsto x\log x.$
On the other hand, it is easy to see that $\nu_-^n$ and $\nu_+^n$ both weakly converge to $\nu_{1/2}$ (this comes from the almost sure convergence of the underlying random variables) and that $\mu^n$ weakly converges to $\mu$. Therefore, by lower semicontinuity of $(\alpha,\beta) \mapsto H(\alpha|\beta)$ for the weak convergence topology, one concludes that 
\[
2H(\nu_{1/2} | \mu) \leq \liminf_{n\to +\infty} \left(H(\nu_-^n | \mu^n) + H(\nu_+^n | \mu^n)\right) \leq H(\nu_0 | \mu) + H(\nu_1 | \mu),
\]
which proves the claim.
\endproof

\section{Inequalities with curvature terms for log-concave distributions.}
Finally, let us show how to derive from Theorem \ref{thm:KL} other versions adapted to log-concave probability measures. The following result is a straightforward restatement of Theorem \ref{thm:KL}.
\begin{cor}\label{cor:KL}
Let $\mu$ be a probability measure on $\Z$ such that $\mu(x)>0$ for all $x \in \Z$.\\
If $f,g,h,k : \Z \to \R^+$ are such that
\[
f(x)g(y) \leq h\left(\left\lfloor \frac{x+y}{2}\right \rfloor\right)k\left(\left\lceil \frac{x+y}{2} \right\rceil\right)e^{c_\mu(x,y)},\qquad \forall x,y\in \Z,
\]
where 
\[
c_\mu(x,y) = \log\left(\frac{\mu\left(\left\lfloor \frac{x+y}{2}\right \rfloor\right)\mu\left(\left\lceil \frac{x+y}{2} \right\rceil\right)}{\mu(x)\mu(y)}\right),\qquad \forall x,y \in \Z,
\]
then it holds
\[
\left(\sum_{x\in \Z} f(x)\mu(x)\right)\left(\sum_{y\in \Z} g(y)\mu(y)\right)\leq \left(\sum_{x\in \Z} h(x)\mu(x)\right)\left(\sum_{y\in \Z} k(y)\mu(y)\right).
\]
\end{cor}
\proof
Simply note that the functions $F(x)=f(x)\mu(x)$, $G(x)=g(x)\mu(x)$, $H(x)=h(x)\mu(x)$ and $K(x)=k(x)\mu(x)$, $x \in \Z$, satisfy the assumptions of Theorem \ref{thm:KL}.
\endproof
Note that the cost function $c_\mu$ always satisfies 
\[
c_\mu(x,x)=0\qquad \text{and}\qquad c_\mu(x,x+1)= c_\mu(x+1,x)=0,\qquad \forall x,y \in \Z.
\]

Let us introduce the optimal transport cost $\mathcal{T}_{c_\mu}$ associated to this cost function $c_\mu$:
\[
\mathcal{T}_{c_\mu} (\nu_0,\nu_1)  = \inf_{\pi \in \Pi(\nu_0,\nu_1)} \iint c_\mu(x,y)\,d\pi(x,y)
\]
with $\Pi(\nu_0,\nu_1)$ the set of probability measures on $\Z^2$ such that the first marginal of $\pi$ is $\nu_0$ and the second is $\nu_1.$

\begin{cor}
Let $\mu$ be a probability measure on $\Z$ such that $\mu(x)>0$ for all $x \in \Z$. Then $\mu$ satisfies the following transport-entropy inequality : for all probability measures $\nu_0,\nu_1$ on $\Z$,
\begin{equation}\label{eq:transport}
\mathcal{T}_{c_\mu} (\nu_0,\nu_1) \leq H(\nu_0|\mu) + H(\nu_1 |\mu).
\end{equation}
\end{cor}
\proof
Let $u,v : \Z \to \R$ be such that
\[
u(x)+v(y) \leq c_\mu(x,y),\qquad \forall x,y \in \Z.
\]
Then according to Corollary \ref{cor:KL} applied to $f=e^u$, $g=e^v$ and $h=k=1$, it holds
\[
\left(\sum_{x\in \Z} e^{u(x)}\mu(x)\right)\left(\sum_{y\in \Z} e^{v(y)}\mu(y)\right)\leq 1.
\]
This is the dual form of \eqref{eq:transport}.
\endproof
The preceding corollary is the most interesting when the cost function $c_\mu$ is non-negative. A natural condition ensuring non-negativity of  $c_\mu$  is the log-concavity of $\mu$. We recall that a probability measure $\mu$ on $\Z$ is log-concave if it is such that
\[
\mu(x-1)\mu(x+1) \leq \mu(x)^2,\qquad \forall x\in \Z.
\]
If one defines, for any $t \in \R$, $V_\mu(t)$ as the linear interpolation between $\log \mu (\lfloor t \rfloor) $ and  $\log \mu (\lceil t \rceil)$, then it is easy to check that $\mu$ is log-concave if and only if the function $V_\mu$ is concave on $\R$.
\begin{lem}
Suppose that $\mu$ is log-concave on $\Z$ and such that $\mu(x)>0$ for all $x\in \Z$, then $c_\mu(x,y) \geq0$ for all $x,y \in \Z.$
\end{lem}
\proof
Without loss of generality, one can assume that $x<y$. If $(x+y)=2k$, with $k \in \Z$, then we have to show that $\mu(k)^2 \geq \mu(x)\mu(y)$. With the notation $V_\mu$ introduced above, this inequality is equivalent to $\frac{V_\mu(k)-V_\mu(x)}{k-x} \geq \frac{V_\mu(y)-V_\mu(k)}{y-k}$ which follows immediately from the concavity of $V_\mu$. If $x+y = 2k+1$, then the inequality $\mu(k)\mu(k+1) \geq \mu(x)\mu(y)$ is equivalent to $\frac{V_\mu(k)-V_\mu(x)}{k-x} \geq \frac{V_\mu(y)-V_\mu(k+1)}{y-(k+1)}$ which again follows from the concavity of $V_\mu$.
\endproof

As an illustration, we end this section with the computation of the  cost $c_\mu$ for two specific examples of probability measures $\mu$ on $\mathbb{Z}$. Consider first the double-sided geometric-type measures $\mu(x)=ce^{-|x|}$, $x \in \mathbb{Z}$, where $c$ is the normalization constant. Then, an easy computation leads to $c_\mu(x,y)=2\min(|x|,|y|)\mathds{1}_{xy <0}$. While for $\mu(x)=ce^{-2x^2}$ (with $c$ again the normalization constant), we get
$c_\mu(x,y)=(x-y)^2\mathds{1}_{x+y \in 2\mathbb{Z}}+[(x-y)^2-1]\mathds{1}_{x+y \in 2\mathbb{Z}+1}$. There is essentially no gain in the first case, which corresponds to a flat situation, while the second example resembles the continuous setting with strictly convex potential for which $\Gamma_2$-calculus applies (see \cite{ane-book,BGL14,Vil2}).

\bibliographystyle{amsplain}
\bibliography{bib-PL}

\end{document}